\documentclass[a4paper,12pt]{article}
\usepackage{amssymb,amsmath,amsfonts,amsthm}
\usepackage{graphicx,afterpage}
\newtheorem{thm}{Theorem}[section]
\newtheorem{lemma}[thm]{Lemma}
\newtheorem{prop}[thm]{Proposition}
\newtheorem{cor}[thm]{Corollary}
\theoremstyle{definition}

\newtheorem{defn}[thm]{Definition}
\begin{document}
    \title{\Large\textbf{How many elements of a Coxeter group have a unique reduced expression?}}
\date{}
      \author{Sarah Hart\thanks{\noindent Dept of Economics, Mathematics and
      Statistics,
Birkbeck College, Malet Street, London, WC1E 7HX. \hspace*{4mm}
s.hart@bbk.ac.uk}}
  \maketitle
 \vspace*{-10mm}
\begin{abstract}
Let $(W,R)$ be an arbitrary Coxeter system. We determine the number of elements of $W$ that have a unique reduced expression.
\end{abstract}
 \section{Introduction}
 
 Given a Coxeter group $W$ with distinguished generating set $R$, every element $w$ of $W$ may be written as a word in $R$. A reduced expression for $w$ is one of minimal length. There are usually several different reduced expressions for any given element. There are results that enable us, in special cases, to count the number of reduced expressions for elements. For example Stanley~\cite{stanley84} gave an algorithm to enumerate the number of reduced expressions for elements of the symmetric group. Eriksson~\cite{eriksson} gave a recursive method for elements of affine Weyl groups. Stembridge investigated the reduced expressions for so-called fully commutative elements \cite{stembridge}. It seems fairly natural to ask about elements that have a unique reduced expression. In this short article we show how to determine very quickly from the Coxeter graph of an arbitrary Coxeter group $W$ the number of elements that have a unique reduced expression. Partial results in this direction are known. For the case of finite Coxeter groups these elements form a 2-sided Kazhdan-Lusztig cell, studied in \cite{cells}. Enumeration of elements with a unique reduced expression for finite Coxeter groups follows from Proposition 4 of that paper and the examples that follow it.

 To state our main result we recall some well-known notation. For more detail on this and other aspects of Coxeter groups see, for example, \cite{humphreys}. A Coxeter system $(W,R)$ is a group $W$ with a generating set $R$ such that $W = \langle R | (rs)^{m_{rs}} = 1; r, s \in R\rangle$, where $m_{rr} = 1$ for all $r \in R$, and $m_{rs} = m_{sr}$. That is, $m_{rs}$ is the order of $rs$. In particular, the elements of $R$ are involutions. We write $m_{rs} = \infty$ where $rs$ has infinite order. A nice way to represent this information is via a Coxeter graph: this is an undirected labelled graph $\Gamma$ with vertex set $R$, where distinct $r, s$ in $R$ are joined by an edge labelled $m_{rs}$ whenever $m_{rs} \geq 3$.  (Usually by convention the label 3 is omitted.) We say that $\Gamma$ is {\em simply laced} if every edge label is 3. Once the generating set $R$ is fixed, then $\Gamma$ is uniquely determined, and in what follows we will assume this has happened. Our technique for counting elements with unique reduced expression relies on an analysis of the Coxeter graph. We remark that other kinds of elements can be counted using the Coxeter graph, such as in the elegant paper by Shi \cite{shi} using the Coxeter graph to enumerate Coxeter elements.
 
 \begin{defn}
 Let $\Gamma$ be a Coxeter graph with associated Coxeter group $W$. We define $U(\Gamma)$ to be the number of words in $W$ with a unique reduced expression. 
 \end{defn}
 
 It will turn out (Lemma \ref{lemma1}) that it is quick to reduce the work to the irreducible case (that is, where the Coxeter graph is connected). We will therefore summarise here our results for the irreducible case. 
 
 \begin{thm}\label{main}
 Suppose $\Gamma$ is an irreducible Coxeter graph with $n$ vertices. \begin{enumerate}
 \item If $\Gamma$ is a simply laced tree, then $U(\Gamma) = n^2 + 1$.
 \item Suppose that $\Gamma$ is a tree with no infinite bonds and exactly one edge with a label $m$ greater than three. Let $a$ and $b$ be the orders of the two induced subgraphs obtained by removal of this edge (so $a + b = n$). Then $$U(\Gamma) = \left\{\begin{array}{ll} \textstyle\frac{1}{2}mn^2 + 1 - 2ab   & {\text{if $m$ even;}}\vspace*{4mm}\\
   \textstyle\frac{1}{2}(m-1)n^2 + 1  & {\text{if $m$ odd.}}
   \end{array}\right.$$
   \item If $\Gamma$ is any other irreducible Coxeter graph then $U(\Gamma) = \infty$.
 \end{enumerate} 
 \end{thm}

 In Section 2 we prove the main results. In Section 3 we give a few example calculations. We finish this section with a final piece of notation. 
 For distinct elements $r$ and $s$ of $R$, we write $[rs]^{n}$ for the (not necessarily reduced) expression with $n$ terms beginning with $r$ and alternating $rsrs\cdots$. So, for example $[rs]^5 = rsrsr$. An {\em elementary operation} on a word consists of replacing $[rs]^{m_{rs}}$ with $[sr]^{m_{rs}}$. It is well known \cite{matsumoto} that any two reduced expressions for an element $w$ of a Coxeter group can be obtained from one another by a sequence of elementary operations.

 \section{Main Results}
 
 In this section we first give in Theorem \ref{thm1} necessary conditions for $U(\Gamma)$ to  be finite. I am grateful to Nathan Reading for reading an earlier version of this paper and pointing out to me that Theorem \ref{thm1} can be deduced from the first part of the proof of \cite[Theorem 5.1]{stembridge}. I have included my proof here for the reader's convenience. These necessary conditions will turn out also to be sufficient conditions. For each case not eliminated by Theorem \ref{thm1} we then find an expression for $U(\Gamma)$, in particular showing that $U(\Gamma)$ is finite. Recall that a {\em chain} in a graph is a path containing at least one vertex that does not contain any cycles. The length of the chain is the number of vertices in the chain.
 
 \begin{thm}\label{thm1}
 Let $\Gamma$ be the Coxeter graph of $W$. Suppose $W$ has finitely many elements with a unique reduced expression. Then $\Gamma$ is finite and each connected component of $\Gamma$ is a tree with no infinite bonds and at most one edge label greater than three. 
 \end{thm}
 \begin{proof} Clearly $\Gamma$ is finite, otherwise $R$ would constitute an infinite set of elements of $W$ each having a unique reduced expression. If $m_{rs} = \infty$ for some $r, s \in R$, then $(rs)^k$ has a unique reduced expression for all positive integers $k$.  If $\Gamma$ contains a cycle then for some $n$ with $n \geq 3$ there are $r_1, \ldots, r_n$ in $R$ for which $m_{r_ir_{i+1}} \geq 3$ when $i < n$ and also $m_{r_nr_1} \geq 3$. Now $(r_1\cdots r_n)^k$ has a unique reduced expression for all positive integers $k$.  This is because any two reduced expressions for a given element can be obtained from each other by a sequence of elementary operations and clearly no elementary operations are possible in this element. We assume from now on that $\Gamma$ has no cycles and no infinite bonds. Suppose that $\Gamma$ contains a chain of the following form, where $m \geq m' \geq 4$.  
 \begin{center} 
 \unitlength 1.00mm
            \linethickness{0.4pt}
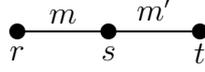
             \begin{picture}(34.00,20)(0,0)
           \put(10,10){\line(1,0){12}}
           \put(10,10.00){\circle*{2.00}}
           \put(22,10.00){\circle*{2.00}}
           \put(34,10.00){\circle*{2.00}}
           \put(10,7){\makebox(0,0)[cc]{$r$}}
           \put(22,7){\makebox(0,0)[cc]{$s$}}
           \put(34,7){\makebox(0,0)[cc]{$t$}}
            \put(16,12){\makebox(0,0)[cc]{$m$}}
             \put(28,13){\makebox(0,0)[cc]{$m'$}}
           \put(22,10){\line(1,0){12}}
         
           \end{picture}\end{center}
Consider $w=srst$. Then $w^k = srstsrstsrst \cdots srst$. Here we do have subexpressions $sts$ and $srs$. But to use an elementary operation we require $[st]^{m'}$, $[ts]^m$, $[rs]^m$ or $[sr]^m$, and so because $m$ and $m'$ are both at least 4, no such transformations are possible. Thus again $w^k$ has a unique expression for all positive integers $k$. 
Finally suppose $\Gamma$ contains a chain of the following form, where $m \geq m' \geq 4$.
 \begin{center} 
 \unitlength 1.00mm
            \linethickness{0.4pt}      
 \begin{picture}(70.00,20)(0,0)
         \put(10,10){\line(1,0){12}}
         \put(22,10.00){\line(1,0){8}}
         \put(50,10.00){\line(1,0){8}}
         \put(58,10){\line(1,0){12}}
        
         \put(58.00,10.00){\circle*{2.00}}
         \put(10,10.00){\circle*{2.00}}
         \put(22,10.00){\circle*{2.00}}
         \put(70,10.00){\circle*{2.00}}
         \put(10,7){\makebox(0,0)[cc]{$r_{1}$}}
         \put(22,7){\makebox(0,0)[cc]{$r_{2}$}}
         \put(58,7){\makebox(0,0)[cc]{$r_{n-1}$}}
         \put(70,7){\makebox(0,0)[cc]{$r_{n}$}}
             \put(16,12){\makebox(0,0)[cc]{$m$}}
                      \put(64,13){\makebox(0,0)[cc]{$m'$}}
         \put(35.00,10.00){\line(1,0){0.4}}
         \put(45.00,10.00){\line(1,0){0.4}}
         \put(40.00,10.00){\line(1,0){0.4}}
         \end{picture}    \end{center}
This time let $w = r_1r_2\cdots r_{n-2}r_{n-1}r_nr_{n-1}r_{n-2}\cdots r_2$. In $w^k$ for $k \geq 1$ the only expressions $[rs]^i$ for any $i$ greater than 2 are $[r_{n-1}r_n]^3$ and $[r_{2}r_1]^3$. However as $m_{r_{n-1}r_n} = m'$ and $m_{r_1r_2} = m$, and both of these are greater than 3, we see that once more no elementary operations are possible. Hence we have infinitely many elements with a unique reduced expression. We conclude that if $W$ only has finitely many such elements, then $\Gamma$ is a forest each of whose connected components is a tree with no infinite bonds and at most one edge label greater than three. 
\end{proof}

\begin{lemma}\label{lemma1}
Suppose $W$ is a Coxeter group with Coxeter graph $\Gamma$ having connected components $\Gamma_1$, \ldots, $\Gamma_n$. Then $U(\Gamma) = \left(\sum_{i=1}^n U(\Gamma_i)\right) - n + 1$.
\end{lemma}
\begin{proof} We have that $W$ is isomorphic to the direct product $W_1 \times \cdots \times W_n$ where for each $i$, the Coxeter group $W_i$ has corresponding Coxeter graph $\Gamma_i$. Suppose a non-identity element $w$ of $W$ has a unique reduced expression. We can write $w$ canonically as  $w = w_1\cdots w_n$ where $w_i \in W_i$, and $w_j \neq 1$ for some $j$ as $w \neq 1$. Then $$w = w_1\cdots w_{j-1}w_j w_{j+1} \cdots w_n = w_jw_1\cdots w_{j-1}w_{j+1}\cdots w_n = w_1\cdots w_{j-1}w_{j+1} \cdots w_nw_j.$$ Since there is a unique reduced expression for $w$, this implies $w_i = 1$ whenever $i \neq j$, and also that $w_j$ has a unique reduced expression in $W_j$. That is, every non-identity element with a unique expression in $W$ is contained in some $W_j$ and has a unique expression in that $W_j$. Clearly every non-identity element of $W_j$ with a unique expression in $W_j$ also has a unique expression in $W$. Therefore there are $\sum_{i=1}^n (U(\Gamma_i) - 1)$ non-identity elements of $W$ that have unique reduced expressions. Hence $U(\Gamma) = \left(\sum_{i=1}^n U(\Gamma_i)\right) - n + 1$.
\end{proof}

We may therefore restrict our attention to the case when $W$ is irreducible, which is equivalent to $\Gamma$ being connected. By Theorem \ref{thm1} we can assume $\Gamma$ is a finite tree with no infinite bonds and at most one edge label $m$ being greater than 3. We require the following easy lemma about chains.

\begin{lemma}
\label{chain} A tree of order $n$ contains precisely $\binom{n}{2}$ chains of length at least 2, and $\binom{n+1}{2}$ chains in total.
\end{lemma}
\begin{proof}
In a tree there is a unique chain between each pair of vertices (otherwise there would be cycles). Therefore there are precisely $\binom{n}{2}$ chains of length at least 2. Adding the $n$ chains of length 1 (each consisting of a single vertex) we see that there are $\binom{n+1}{2}$ chains in total.
\end{proof}

\begin{prop}\label{prop}
Suppose $\Gamma$ is a simply laced tree with $n$ vertices for some positive integer $n$. Then $U(\Gamma) = n^2 + 1$. In particular, $U(\Gamma)$ is finite. 
\end{prop}
\begin{proof}
Suppose $w$ has a unique reduced expression $r_1 \cdots r_k$ for some (not necessarily distinct) $r_i \in R$. Moreover, if any $r_i$ commutes with $r_{i+1}$, for $1 \leq i < k$, then $w$ would have another reduced expression $r_1\cdots r_{i-1}r_{i+1}r_ir_{i+2}\cdots r_k$. Hence $m_{r_ir_{i+1}} = 3$ for all $i$. Suppose $r_i = r_j$ for some $i < j$, and let us assume $|j-i|$ is minimal such that this occurs. That is, $r_i, r_{i+1}, \ldots r_{j-1}$ are all distinct elements of $R$. Obviously $j = i+1$ is impossible as this is a reduced expression. If $j = i+2$ then we have $r_ir_{i+1}r_i$ as a subexpression of $w$. But $\Gamma$ is simply laced, meaning $r_{i}r_{r+1}r_i = r_{i+1}r_ir_{i+1}$, contradicting the uniqueness of the reduced expression for $w$. Therefore $j > i+2$. But then the vertices $r_{i}, r_{i+1}, \ldots, r_{j-1}$ form a cycle of $\Gamma$, contradicting the fact that $\Gamma$ is a tree. Therefore in fact the $r_i$ are all distinct. Hence $\Gamma_w$, the induced subgraph whose vertex set is $r_1, \ldots, r_k$, is the following chain.
 \begin{center} 
 \unitlength 1.00mm
            \linethickness{0.4pt}      
 \begin{picture}(70.00,20)(0,0)
         \put(10,10){\line(1,0){12}}
         \put(22,10.00){\line(1,0){8}}
         \put(50,10.00){\line(1,0){8}}
         \put(58,10){\line(1,0){12}}
        
         \put(58.00,10.00){\circle*{2.00}}
         \put(10,10.00){\circle*{2.00}}
         \put(22,10.00){\circle*{2.00}}
         \put(70,10.00){\circle*{2.00}}
         \put(10,7){\makebox(0,0)[cc]{$r_{1}$}}
         \put(22,7){\makebox(0,0)[cc]{$r_{2}$}}
         \put(58,7){\makebox(0,0)[cc]{$r_{k-1}$}}
         \put(70,7){\makebox(0,0)[cc]{$r_{k}$}}
         \put(35.00,10.00){\line(1,0){0.4}}
         \put(45.00,10.00){\line(1,0){0.4}}
         \put(40.00,10.00){\line(1,0){0.4}}
         \end{picture}    \end{center} 
         Every element $w$ with a unique reduced expression corresponds to a unique chain of $\Gamma$. However $w^{-1}$ produces the same chain, and $w^{-1} = w$ if and only if $n \leq 1$. Therefore each chain of length at least two produces exactly two elements with unique reduced expressions. By Lemma \ref{chain} there are $\binom{n}{2}$ chains of length at least 2, each providing two elements with a unique reduced expression.  Each of the $n$ vertices (chains of length 1) provides exactly one element with a unique reduced expression. Adding the identity element we therefore see that $U(\Gamma) = 2\binom{n}{2} + n + 1 = n^2 + 1$. 
\end{proof}

\begin{thm}
\label{m}
Suppose $\Gamma$ is a tree with $n$ vertices, no infinite bonds and exactly one edge with a label $m$ greater than three. Let $a$ and $b$ be the orders of the two induced subgraphs obtained by removal of this edge (so $a + b = n$). Then $$U(\Gamma) = \left\{\begin{array}{ll} \textstyle\frac{1}{2}mn^2 + 1 - 2ab   & {\text{if $m$ even;}}\vspace*{4mm}\\
\textstyle\frac{1}{2}(m-1)n^2 + 1  & {\text{if $m$ odd.}}
\end{array}\right.$$
\end{thm}

\begin{proof}
Let $r$ and $s$ be the vertices of $\Gamma$ which are joined by the edge labelled $m$. Consider the subgraph induced by removing the edge $m$. Let $\Delta$ be the connected component containing $r$ and $\Sigma$ be the connected component containing $s$. Both $\Delta$ and $\Sigma$ are simply laced finite trees. Set $a = |\Delta|$ and $b = |\Sigma|$. Let $w$ be a non-identity element of $W$ that has a unique reduced expression, and let $\Gamma_w$ be the subgraph of $\Gamma$ induced by the elements of $R$ contained in the expression for $W$.  Writing $w = r_1\cdots r_k$ for some $r_i \in R$, we observe that for each $i$ in $\{1, \ldots, k=1\}$, we have that $r_i$ and $r_{i+1}$ are distinct (otherwise the expression would not be reduced) and moreover there is an edge between $r_i$ and $r_{i+1}$ in $\Gamma_w$, otherwise $r_i$ would commute with $r_{i+1}$, implying the existence of a second reduced expression. Therefore $\Gamma_w$ is connected. Suppose first that $\Gamma_w$ does not contain the edge labelled $m$. Then $\Gamma_w$ is a simply-laced tree and, as in the argument for Proposition \ref{prop}, $\Gamma_w$ is in fact a chain and $r_1, \ldots r_k$  are all distinct. Also $\Gamma_w$ must be contained in either $\Delta$ or $\Sigma$. By Lemma \ref{chain} there are $\binom{a}{2}$ chains of length at least 2 in $\Delta$ and $\binom{b}{2}$ chains of length at least 2 in $\Sigma$. Each of these results in two elements ($w$ and $w^{-1}$) with a single reduced expression. Each of the vertices of $\Gamma$ gives one further element. Therefore the total number of non-identity elements $w$ with a unique reduced expression such that $\Gamma_w$ does not contain the edge labelled $m$ is $2(\binom{a}{2} + \binom{b}{2}) + n$, which, remembering that $a+b = n$, is equal to $a^2 + b^2$. \\

Now we consider the case where $\Gamma_w$ does contain the edge labelled $m$. Suppose that there are $i, j$ with $1 \leq i < j \leq k$ such that $\{r_i, r_j\} \subseteq \{r,s\}$, and $\{r_{i+1}, \ldots, r_{j-1}\} \cap \{r, s\} = \emptyset$. Let $u = r_ir_{i+1}\cdots r_{j-1}$. Then $u$ is an element with a unique reduced expression; moreover $\Gamma_u$ does not contain the edge labelled $m$. Thus $\Gamma_u$ is a chain and the elements $r_i,r_{i+1},\ldots, r_{j-1}$ are all distinct. But $r_{j-1}$ is adjacent in $\Gamma_w$ to $r_j$ which is either $r$ or $s$. Either way, it implies that there is a cycle in $\Gamma$, a contradiction. Suppose for the moment that $r$ appears before $s$ in $w$. Then $w$ is of the form $r_1 \cdots r_{i}[rs]^Ls_1 \cdots s_j$ where $r_1, \ldots, r_i, s_1, \ldots, s_j \in R\setminus\{r,s\}$. To preserve the uniqueness of the expression, we must have $L < m$, and to ensure that $\Gamma_w$ contains the edge labelled $m$, we also know that $L \geq 2$. 
Moreover $r_1, \ldots, r_{i}, r$ is a chain in $\Delta$. If $L$ is even then $s, s_1, \ldots, s_j$ is a chain in $\Sigma$. If $L$ is odd then $rs_1, \ldots, s_j$ is a chain in $\Delta$. 

Suppose that $L$ is even. Then $\Gamma_w$ is a chain between an element of $\Delta$ and an element of $\Sigma$ and each such $w$ results in exactly one such chain. This chain will also arise from $w^{-1}$, which is an element where $s$ appears before $r$. Therefore for each even $L$ lying between $2$ and $m-1$, each of the $ab$ chains between elements of $\Delta$ and elements of $\Sigma$ results in exactly two elements having unique reduced expressions (one where $r$ appears before $s$, one where $s$ appears before $r$). Therefore there are $2ab\lfloor\frac{m-1}{2}\rfloor$ such elements. 

Now suppose that $L$ is odd, and for the moment that $r$ appears before $s$ in the expression for $w$. This means $r$ also appears before $s$ in the expression for $w^{-1}$. If $r_1 = s_j$ then $w$ is an involution and $r_1, \ldots, r_i, r$ is a chain in $\Delta$ from $r_1$ to $r$. There are $a$ chains  in $\Delta$ ending in $r$, each producing exactly one such involution $w$. Hence there are $a$ elements of this form for each odd $L$ (similarly there are $b$ involutions in which $s$ appears first). If $r_1 \neq s_j$ then $w$ is not an involution, so $\Gamma_w$ corresponds to two elements, $w$ and $w^{-1}$, both of which have the property that $r$ appears first in the reduced expression. Since $r_1,\ldots, r_i, r$ and $r, s_1, \ldots, s_j$ are chains in $\Delta$, the number of such elements $w$ is twice the number of ways of choosing two different chains ending at $r$, because each such pair of chains results in two elements, $w$ and $w^{-1}$. So we get $2\binom{a}{2}$ elements $w$ which is just $a(a-1)$. Similarly for each odd $L$ there are $b(b-1)$ elements $w$ where $s$ appears before $r$. So for each odd $L$ the total number of elements $w$ with a unique reduced expression is $a + b + a(a-1) + b(b-1) = a^2 + b^2$. Summing over the odd $L$ between 2 and $m-1$ we get $(a^2 + b^2)\lfloor\frac{m-2}{2}\rfloor$ elements. \\

Combining the calculations for $L$ even and $L$ odd, we see that the total number of elements $w$ with a unique reduced expression such that $\Gamma_w$ contains the edge labelled $m$ is $$2ab\left\lfloor\frac{m-1}{2}\right\rfloor + (a^2 + b^2)\left\lfloor\frac{m-2}{2}\right\rfloor.$$  To obtain $U(\Gamma)$ we must add to this the identity element plus the $a^2 + b^2$ non-identity elements $w$ for which $\Gamma_w$ does not contain the edge labelled $m$. If $m$ is even then, recalling that $a+b = n$, we get $$U(\Gamma) = a^2 + b^2 + 1 + \textstyle\frac{1}{2}(2ab + a^2 + b^2)(m-2)  = n^2 - 2ab + 1 + \textstyle\frac{1}{2}n^2(m-2) = \frac{1}{2}mn^2 + 1 - 2ab.$$ If $m$ is odd then we get \begin{align*} U(\Gamma) &= a^2 + b^2 + 1 + \textstyle\frac{1}{2}\left((2ab)(m-1) + (a^2 + b^2)(m-3)\right) \\ &= a^2 + b^2 + 1 + 2ab + \textstyle\frac{1}{2}(a^2 + b^2 + 2ab)(m-3) \\ &= \textstyle\frac{1}{2}(m-1)n^2 + 1.\qedhere\end{align*} 
\end{proof}

Theorem \ref{thm1}, Proposition \ref{prop} and Theorem \ref{m} combine to give Theorem \ref{main}, along with the following corollary, which classifies the Coxeter groups having finitely many elements with a unique reduced expression. 

\begin{cor}
Let $\Gamma$ be the Coxeter graph of $W$. Then $W$ has finitely many elements with a unique reduced expression if and only if $\Gamma$ is finite and each connected component of $\Gamma$ is a tree with no infinite bonds and at most one edge label greater than three.
\end{cor}

\section{Examples}

In this section we give some example calculations. Proposition \ref{prop} deals with all simply laced Coxeter graphs $\Gamma$: in each case there are $|\Gamma|^2 + 1$ elements with a unique reduced expression. So for example there are 67 such elements in the Coxeter groups of types $A_8$, $D_8$ and $E_8$ (and indeed any simply-laced Coxeter group of rank 8). For a group of type $B_n$ we have  $m=4$, $a=1$ and $b=n-1$. So by Theorem \ref{m} there are 
$2n^2 + 1 - 2(n-1)$ elements with a unique reduced expression, which is $2n(n-1) + 3$. So in $B_4$ there are 27 such elements, for example. In $F_4$ we have $a = b = 2$ and $U(F_4) = 25$. Below is a table listing $U(\Gamma)$ for each irreducible finite and affine Coxeter group.
\begin{center}
\begin{tabular}{|l|c|c|c|} \hline $\Gamma$ & $U(\Gamma)$ & $\Gamma$ & $U(\Gamma)$\\
\hline
$A_n (n \geq 1)$ & $n^2 + 1$ & $\tilde A_n (n \geq 1)$ & $\infty$\\
$B_n (n \geq 2)$ & $2n^2 - 2n + 3$ & $\tilde B_n (n \geq 3)$ & $2n^2 + 2n + 3$ \\
$D_n (n \geq 4)$ & $n^2 + 1$  & $\tilde C_n (n \geq 2)$ & $\infty$\\
$E_6$ & 37 & $\tilde D_n (n \geq 4)$ & $(n+1)^2 + 1$\\
    $E_7$ & 50  & $\tilde E_6$ & $50$\\
$E_8$ & 65 & $\tilde E_7$ & $65$\\
$F_4$ & 25  & $\tilde E_8$ & $82$\\
$I_2(m) (m \geq 6)$ & $2m-1$ & $\tilde F_4$ & $39$\\
$H_3$ & $19$  & $\tilde G_2$ & 24\\
  $H_4$ & $33$ && \\ \hline
\end{tabular}
\end{center}

 \end{document}